\documentclass{amsart}

\usepackage[utf8]{inputenc}
\usepackage{amsmath}
\usepackage{amssymb}
\usepackage{amsthm}
\usepackage{stmaryrd}
\usepackage{pgfplots}
\usepackage{pgfplotstable}
\usepackage{tikz}
\usepackage{algorithm}
\usepackage{xcolor}
\usepackage{algorithmicx}
\usepackage{algpseudocode}

\newtheorem{problem}{Problem}
\newtheorem{thm}{Theorem}
\newtheorem{lem}{Lemma}
\theoremstyle{definition}
\newtheorem{rem}{Remark}

\pgfplotsset{compat=1.17}

\title[Stabilized finite elements for Tresca friction problem]{Stabilized finite elements for\\Tresca friction problem}

\newcommand{\Div}[0]{\mathbf{div}\,}
\newcommand{\bs}[1]{\boldsymbol{#1}}
\newcommand{\Sig}[0]{\bs{\sigma}}

\newcommand{\Eps}[0]{\bs{\varepsilon}}
\newcommand{\U}[0]{\bs{u}}
\newcommand{\V}[0]{\bs{v}}
\newcommand{\W}[0]{\bs{w}}
\newcommand{\Lam}[0]{\bs{\lambda}}
\newcommand{\Mu}[0]{\bs{\mu}}
\newcommand{\X}[0]{\bs{\xi}}

\newcommand{\Bf}[0]{\mathcal{B}}
\newcommand{\Lf}[0]{\mathcal{L}}

\newcommand{\N}[0]{\bs{n}}
\newcommand{\T}[0]{\bs{t}}
\newcommand{\F}[0]{\bs{f}}
\newcommand{\Z}[0]{\bs{0}}
\newcommand{\Th}[0]{\mathcal{T}_h}
\newcommand{\Eh}[0]{\mathcal{E}_h}
\newcommand{\Fh}[0]{\mathcal{F}_h}
\newcommand{\Gh}[0]{\mathcal{G}_h}
\newcommand{\enorm}[1]{{\left\vert\kern-0.25ex\left\vert\kern-0.25ex\left\vert #1 
            \right\vert\kern-0.25ex\right\vert\kern-0.25ex\right\vert}}

\begin{document}

\begin{abstract}
We formulate and analyze a Nitsche-type algorithm for  frictional contact problems. The method is derived from, and analyzed as, a stabilized finite element method and shown to be  quasi-optimal, as well as suitable as an adaptive scheme through an a posteriori error analysis. The a posteriori error indicators are validated in a numerical experiment. 
\end{abstract}
\author[Tom Gustafsson]{Tom~Gustafsson}
\author[Juha Videman]{Juha~Videman}

\address[Tom Gustafsson]{Aalto University, Department of Mathematics and Systems Analysis, P.O. Box 11100 FI-00076 Aalto, Finland}
\address[Juha Videman]{CAMGSD and Mathematics Department, Instituto Superior T\'ecnico, Universidade de Lisboa, Av. Rovisco Pais 1, 1049-001 Lisboa, Portugal}

\maketitle

\section{Introduction}

The frictionless contact between deforming bodies
can be interpreted as a minimization problem
with a nonpeneration constraint imposed on the displacement field, see, \emph{e.g.}, \cite{KO88,HHNBook,Wriggers}.
When the friction between the bodies is taken into account, the contact is usually modelled through a solution-dependent  upper bound function for
the tangential traction (Coulomb friction model), see, \emph{e.g.}, \cite{DL76, KO88}. If the upper bound function is prescribed,
the Coulomb friction model simplifies to the  Tresca  friction problem.

In both friction models, the constraints can be resolved with the help of Lagrange multipliers. The  variational inequalities corresponding to the mixed formulation of these contact problems have a saddle point
structure and their approximation by standard finite elements
leads to unstable and ill-conditioned methods.
In particular, the methods are stable if and only if the chosen
finite element spaces satisfy the Babu\v{s}ka--Brezzi inf-sup condition
which is nontrivial in the general case of nonmatching meshes.
Consequently, special finite element bases have been developed
for the Lagrange multiplier that address the stability
and lead to optimally convergent methods, see \cite{W11} and all the references therein.

As an alternative to a mixed method, we consider a stabilized finite element method. To keep the notation and the presentation simple and readable, we restrict ourselves to the unilateral contact between a deformable body and a rigid foundation and consider the Tresca friction problem. We note, however, that it is  straightforward to extend our method to  contact between two deformable bodies.
Besides avoiding the Babu\v{s}ka--Brezzi inf-sup condition, the stabilized method provides an elegant way to derive and justify the use of Nitsche's method for approximating frictional contact problems, cf.~\cite{Chouly2014,CMR18,CFHPR18}. Most importantly, building upon the stabilized formulation we are able to prove the efficiency and reliability of residual-based a posteriori error estimators for Nitsche's method without extra regularity assumptions or resorting to a saturation assumption as in \cite{CFHPR18}.

In this work, we capitalize on the  analysis of stabilized finite element methods for variational inequalities, first presented in \cite{GSV2017} and recently extended to frictionless contact problems in \cite{GSV2020}.
We will  omit some of the proofs (\emph{e.g.}, the continuous  stability estimate) that would follow step by step the reasoning presented in detail in \cite{GSV2017, GSV2020} and instead put more emphasis on the a posteriori error indicators in adaptive schemes as well as discuss the practical implementation of the method  as a Nitsche-type numerical algorithm.
The algorithm is validated in a numerical experiment using open source software
and the source code is freely available~\cite{zenodo}.

The numerical approximation of frictional contact between deformable bodies  has a long-lived history, see, \emph{e.g.}, \cite{KO88,Laursen, Wriggers,W11} and all the references therein.  Different approaches abound, ranging from mixed saddle-point formulations based on dual Lagrange multipliers  
(\cite{HH82,BailletSassi,HS2004,HildRenard2007,HSW2008,W11}) to  Nitsche-type method (\cite{Chouly2014,CFHPR18,CEP2020}).
For a posteriori error analyses of the frictional contact problem based on the saddle-point formulation, we refer to \cite{HildLleras,DM2010,W11} and based on Nitsche's formulation to the Appendix in \cite{CFHPR18}.

The article is organized as follows. In Section 2, we explain the problem setting, in Section 3 give a variational formulation to the Tresca friction  problem, in Section 4 formulate our finite element scheme and in Section 5 show quasi-optimality of the method and perform the a posteriori error analysis. In Section 6, we derive Nitsche's method from the stabilized method, introduce an algorithm for the practical computation of the numerical solution and discuss  implementional aspects. Finally, in Section 7 we present a numerical experiment to corroborate  the usefulness of the a posteriori error indicators.

\section{Tresca friction problem}
Let $\Omega \subset \mathbb{R}^d$, $d\in\{2,3\}$, denote a deformable polygonal (polyhedral) body.
The boundary $\partial \Omega $  is split into three parts
$\Gamma_{D}$, $\Gamma_{N}$ and $\Gamma$, with $\Gamma_{D}$ denoting the part where the displacement of the body is zero, $\Gamma_{N}$ the part of the boundary  with zero traction  and $\Gamma$ the part where contact between the body and a rigid foundation can occur.
The parts $\Gamma_{D}$ and $\Gamma$ are assumed to be separated by $\Gamma_{N}$, i.e.~$\overline{\Gamma_{D}} \cap \overline{\Gamma} = \varnothing$, and $\Gamma$ is a straight line if $d=2$ or a planar polygon if $d=3$.

Let $\U : \Omega \rightarrow \mathbb{R}^d$, denote the displacement of the body $\Omega$.
The infinitesimal strain tensor reads
\begin{equation}
    \Eps(\U) = \frac12\Big(\nabla \U + (\nabla \U)^T\Big),
\end{equation}
and the stress tensor is given by
\begin{equation}
    \Sig(\U) = 2 \mu\,\Eps(\U) + \lambda\,\mathrm{tr}\,\Eps(\U) \boldsymbol{I},
\end{equation}
where $\mu$ and $\lambda$ are the Lam\'e parameters and $\boldsymbol{I}$
is an identity tensor.
We define the normal component of $\U$ as $u_n = \U \cdot \N$ 
where $\N : \partial \Omega \rightarrow \mathbb{R}^{d}$ denotes the outward unit normal to the body $\Omega$.
The traction vector $\Sig(\U)\N$
is split into the normal component $\sigma_n(\U)\N$, $\sigma_{n}(\U) = \Sig(\U) \N \cdot \N$,
and the tangential component $\Sig_{t}(\U) = \Sig(\U)\N - \sigma_{n}(\U)\N$.  Similarly, we define the tangential
displacement as $\U_t = \U - u_n\N$.

The governing equation and the boundary conditions read as follows
\begin{alignat}{2}
    -\Div \Sig(\U) &= \F \quad && \text{in $\Omega$},\\
    \U &= \Z \quad && \text{on $\Gamma_{D}$,}\\
    \Sig(\U)\N &= \Z \quad && \text{on $\Gamma_{N}$,}
\end{alignat}
where $\F \in [L^2(\Omega)]^d$ is the volumetric force.
The physical nonpenetration condition on $\Gamma$ is
\begin{equation}
    u_n - g \leq 0, \quad \sigma_{n}(\U) \leq 0, \quad \sigma_{n}(\U) ( u_n - g) = 0,
\end{equation}
where $g \in H^{1/2}(\Gamma)$ is the gap between
the body and a rigid foundation in the direction of $\N$.
The Tresca friction condition on $\Gamma$ reads as (cf. \cite{DL76,KO88})
\begin{equation}
    \begin{cases}
        | \Sig_{t}(\U) | \leq \kappa, \\
        | \Sig_{t}(\U) | < \kappa \Rightarrow \U_t = \Z,\\
        | \Sig_{t}(\U) | = \kappa > 0 \Rightarrow \exists \nu \geq 0 : \U_t = - \nu \Sig_{t}(\U),
    \end{cases}
\end{equation}
where $\kappa\in L^2(\Gamma)$, $\kappa \geq 0$ a.e.~in $\Gamma$, is a given upper limit for the tangential traction before slip can occur.

The Tresca friction problem can be written as a mixed problem by introducing a dual variable (Lagrange multiplier) $\Lam = - \Sig(\U)\N$, and considering separately its normal and tangential components, $\lambda_n = - \sigma_n(\U)$
and  $\Lam_t = \Lam - \lambda_n \N$.

\begin{problem}[mixed formulation]
\label{prob:mixed}
    Find $(\U, \Lam)$ satisfying
    \begin{alignat}{2}
    -\Div \Sig(\U) &= \F \quad && \text{in $\Omega$},\\
    \U &= \Z \quad && \text{on $\Gamma_{D}$,}\\
    \Sig(\U)\N &= \Z \quad && \text{on $\Gamma_{N}$,}\\
    \Lam + \Sig(\U)\N &= \Z \quad && \text{on $\Gamma$,} \label{eq:stabres}
\end{alignat}
with the contact conditions
\begin{equation}
\begin{cases}
    \lambda_n \geq 0,\\
    u_n -g \leq 0,\\
    \lambda_n(u_n - g) = 0,
\end{cases}
\text{on $\Gamma$},
\end{equation}
and the friction conditions
\begin{equation}
\begin{cases}
    | \Lam_t | \leq \kappa,\\
    | \Lam_t | < \kappa \Rightarrow \U_t = \Z,\\
    | \Lam_t | = \kappa > 0 \Rightarrow \exists \nu \geq 0 : \U_t = \nu \Lam_t,
\end{cases}
\quad \text{on $\Gamma$}.
\end{equation}
\end{problem}

\section{Variational formulation}
We will now present a variational formulation for Problem~\ref{prob:mixed}.
For the primal variable, we consider the  standard Sobolev space
\begin{equation}
    \bs{V} = \{\W \in [H^1(\Omega)]^d : \W |_{\Gamma_{D}} = \Z\}.
\end{equation}
To introduce the spaces for the functions defined on $\Gamma$, we assume the existence of a local orthonormal basis $\{\N(\bs{x}), \T(\bs{x})\}$ for $d=2$
and $\{\N(\bs{x}), \T_1(\bs{x}), \T_2(\bs{x})\}$ for $d=3$ at $\bs{x} \in \Gamma$, and
split the trace $\W|_\Gamma$ into normal and tangential components, i.e.
\[
\W(\bs{x}) = (w_n(\bs{x}), \W_t(\bs{x})) \in \mathbb{R} \times \mathbb{R}^{d-1}
\]
for any $\bs{x} \in \Gamma$.
Consequently, the trace of $\W \in \bs{V}$ on $\Gamma$ belongs to $H^{1/2}(\Gamma) \times (H^{1/2}(\Gamma))^{d-1}$
with the norm
\[
    \|(w_n, \W_t)\|^2_{\frac12,\Gamma} = \|w_n\|_{\frac12,\Gamma}^2 + \|\W_t\|_{\frac12,\Gamma}^2
\]
where
\[
\|w\|_{\frac12,\Gamma} = \inf_{v \in H^1(\Omega), v|_{\Gamma} = w} \|v\|_1, \quad \text{and} \quad 
    \|\W\|^2_{\frac12,\Gamma} = \sum_{i=1}^{d-1} \|w_i\|_{\frac12,\Gamma}^2.
\]
This splitting allows treating separately the normal and the tangential components of the Lagrange multiplier vector which in the Tresca friction formulation are uncoupled. To this end, we define the spaces
\begin{align*}
    \varLambda_n &= \{ \mu \in H^{-1/2}(\Gamma) : \langle v, \mu \rangle \geq 0 ~\forall v \in H^{1/2}(\Gamma), v \geq 0~\text{a.e.~in $\Gamma$} \}, \\
    \bs{\varLambda}_t &= \{ \Mu \in (H^{-1/2}(\Gamma))^{d-1} : \langle \V, \Mu \rangle \leq ( \kappa, | \V |)_\Gamma~\forall \V \in (H^{1/2}(\Gamma))^{d-1} \},
\end{align*}
where by $\langle \cdot,\cdot \rangle \!: \!H^{1/2} (\Gamma ) \times  H^{-1/2}(\Gamma)\rightarrow  \mathbb{R}$ we denote the duality pairing and, for $(\W,\X) \in (H^{1/2}(\Gamma))^{d-1} \times (H^{-1/2}(\Gamma))^{d-1}$,  write 
$\langle \W, \X \rangle = \sum_{i=1}^{d-1} \langle w_i, \xi_i \rangle$.  
The dual space $H^{-1/2}(\Gamma)$ and the vectorial dual space $(H^{-1/2}(\Gamma))^{d-1}$ are equipped with the norms
\[
\|\xi\|_{-\frac12,\Gamma} = \sup_{v \in H^{1/2}(\Gamma)} \frac{\langle v, \xi \rangle}{\|v\|_{\frac12,\Gamma}} \quad \text{and} \quad
\| \X \|_{-\frac12,\Gamma}^2 = \sum_{i=1}^{d-1} \| \xi_i\|_{-\frac12,\Gamma}^2,
\]
respectively.
Moreover, we let $\bs{Q}=H^{-1/2}(\Gamma) \times (H^{-1/2}(\Gamma))^{d-1}$ and define
a bilinear form
$\Bf : (\bs{V} \times \bs{Q}) \times (\bs{V} \times \bs{Q}) \rightarrow \mathbb{R}$
through
\begin{equation}
\Bf(\W, \X; \V, \Mu) = (\Sig(\W), \Eps(\V))_\Omega + b( \V, \X ) + b( \W, \Mu ),
\end{equation}
where $b(\W, \X) = \langle w_n, \xi_n \rangle + \langle \W_t, \X_t \rangle$ for any $\W \in \bs{V}$ and $\X = (\xi_n, \X_t) \in \bs{Q}$.
We also define a linear form $\Lf : \bs{V} \times \bs{Q} \rightarrow \mathbb{R}$ by
\begin{equation}
\Lf(\V, \Mu) = (\F, \V)_{\Omega} + \langle g, \mu_n \rangle.
\end{equation}

Let $\bs{\varLambda} = \varLambda_n \times \bs{\varLambda}_t \subset \bs{Q}$. It is now straightforward to write Problem~\ref{prob:mixed} as the following variational inequality.  
\begin{problem}[variational formulation]
   \label{prob:weak}
   Find $(\bs{u}, \bs{\lambda}) \in \bs{V} \times \bs{\varLambda}$ such that
   \[
       \Bf(\U,\Lam;\V,\Mu - \Lam) \leq \Lf(\V, \Mu - \Lam) \quad \forall (\V, \Mu) \in \bs{V} \times \bs{\varLambda}.
   \]
\end{problem}

\begin{rem} The corresponding primal variational formulation of the Tresca friction problem  reads as follows:
   find $\bs{u} \in \bs{K}=\{\V \in \bs{V}:v_n-g\leq 0~\text{a.e.~on $\Gamma$} \}$ such that
   \[
       (\Sig(\U), \Eps(\V-\U))_\Omega +( \kappa,|\V_t|)_\Gamma - ( \kappa,|\U_t|)_\Gamma \leq  (\F, \V-\U)_{\Omega} \quad \forall \V \in \bs{K}.
   \]
 The existence of a unique solution to this variational inequality of the second kind is classical, cf.~\cite{DL76,KO88,HHNBook}. However, the nondifferentiable term  $(\kappa,|\V_t|)_\Gamma$ makes the primal formulation less amenable to numerical approximation than the mixed formulation.
 The equivalence between the mixed and the primal formulations has been established in \cite{HH82}, see also \cite{BailletSassi,HS2004}.
\end{rem}

\begin{rem}
Since $\kappa\in L^2(\Gamma)$ and the contact region is smooth, the tangential traction is in fact a $L^2(\Gamma)$-function, cf.~\cite{LabRen2008}, and the set $ \bs{\varLambda}_t$ is often defined  using this extra regularity, cf.~\cite{HHNBook,BailletSassi,HS2004}. The error estimates for $\Lam_t$ are, however, naturally obtained in the $H^{-1/2}(\Gamma)$-norm and the definition  chosen here for $ \bs{\varLambda}_t$ can be naturally extended to the Coulomb friction case, cf.~\cite{W11}. Moreover, the stabilization terms defined below are related with discrete $H^{-1/2}(\Gamma)$-norms and the resulting stabilized method leads to the  Nitsche's method presented in the literature \cite{Chouly2014,French-survey,CMR18}. 
\end{rem}

The variational formulation will be analyzed in the norm
\begin{equation}
\|(\W,\X) \|^2 = \enorm{\W}^2 + \|\xi_n\|^2_{-\frac12,\Gamma} + \|\X_t\|^2_{-\frac12,\Gamma}
\end{equation}
where
\begin{equation}
    \enorm{\W}^2 = (\Sig(\W), \Eps(\W))_\Omega.
\end{equation}
The proof of the following result is quite standard, see, \emph{e.g.},  \cite{GSV2020}, for more details. 

\begin{thm}[continuous stability]
   For every $(\W,\X) \in \bs{V} \times \bs{Q}$ there exists $\V \in \bs{V}$ and constants $C_1,C_2>0$ such that
   \begin{equation}
       \Bf(\W,\X;\V,-\X) \geq C_1 \|(\W, \X)\|^2 
   \end{equation}
   and
   \begin{equation}
        \enorm{\V} \leq C_2 \|(\W, \X)\|.
   \end{equation}
\end{thm}

\section{Finite element method}

Let $\Th$ denote a subdivision of the domain $\Omega$ into
nonoverlapping elements, triangles ($d=2$) or tetrahedra ($d=3$),
with $h$ denoting the global mesh parameter.
We denote by $\Gh$ the trace mesh of $\Th$ on $\Gamma$,
consisting of edge segments ($d=2$) or facet triangles ($d=3$)
of the elements in $\Th$.  The set of interior edges/facets are denoted by $\Eh$ and the set of boundary edges/facets belonging to the boundary $\Gamma_N$  by $\Fh$. In what follows, we write $a \lesssim b$ if $a \leq C b$ for some $C>0$ independent of the mesh parameter $h$; similarly for $a \gtrsim b$.

Aiming at constructing a uniformly stable approximation of Problem \ref{prob:weak}, 
we augment the bilinear form $\Bf$ using the residual of \eqref{eq:stabres}
as follows:
\begin{equation}
\Bf_h(\W, \X; \V, \Mu) = \Bf(\W, \X; \V, \Mu) - \alpha \sum_{E \in \Gh} h_E (\X + \Sig(\W)\N, \Mu + \Sig(\V)\N)_E
\end{equation}
where $\alpha > 0$ is a stabilization parameter and $h_E$ is the local mesh parameter corresponding to $E \in \Gh$.
The finite element spaces $\bs{V}_h \subset \bs{V}$ and $\bs{Q}_h \subset \bs{Q}$ are defined as
\begin{align}
    \bs{V}_{h} &= \{ \W \in \bs{V} : \W|_K \in [P_m(K)]^d~\forall K \in \Th \}, \\
    \bs{Q}_h &= \{ (\mu_n, \Mu_t) \in \bs{Q} : \mu_n|_E \in P_l(E)~\text{and}~\Mu_t|_E \in [P_l(E)]^{d-1}~\forall E \in \Gh \},
\end{align}
where $m \geq 1$ and $l \geq 0$ denote the polynomial orders.
In addition, we define the convex subset
\begin{equation}
    \bs{\varLambda}_h = \{ (\mu_{n}, \Mu_{t}) \in \bs{Q}_h : \mu_{n} \geq 0, ~ | \Mu_{t} | \leq \kappa \} \subset \bs{\varLambda}.
\end{equation}

The stabilized finite element method corresponds to solving the following variational problem. 
\begin{problem}[discrete formulation]
  \label{prob:stab}
   Find $(\U_h, \Lam_h) \in \bs{V}_h \times \bs{\varLambda}_h$ such that
  \begin{equation}
       \Bf_h(\U_h,\Lam_h;\V_h,\Mu_h - \Lam_h) \leq \Lf(\V_h, \Mu_h - \Lam_h) \quad \forall (\V_h, \Mu_h) \in \bs{V}_h \times \bs{\varLambda}_h.
       \label{discprob}
\end{equation}
\end{problem}
We will show in Theorem \ref{thm:disc} below that the discrete formulation
is stable in the following mesh-dependent norm:
\begin{equation}
\| (\V,\Mu) \|^2_h = \|(\V,\Mu)\|^2 + \sum_{E \in \Gh} h_E \|\Mu\|_{0,E}^2, \quad (\V,\Mu) \in \bs{V}_h \times \bs{Q}_h.
\end{equation}
In the proof, we will use the following  discrete trace estimate, established using a scaling argument.
\begin{lem}
\label{lem:trace}
For any $\W_h \in \bs{V}_h$, there exists $C_I>0$ such that 
\begin{equation}
    C_I \sum_{E \in \Gh} h_E \| \Sig(\W_h)\N \|_{0,E}^2 \leq \enorm{ \W_h }^2.
\end{equation}
\end{lem}

\begin{thm}[discrete stability]
\label{thm:disc}
Suppose that $0 < \alpha < C_I$.
Then for every $(\W_h,\X_h) \in \bs{V}_h \times \bs{Q}_h$ there exists $\V_h \in \bs{V}_h$ satisfying
\begin{equation}
    \Bf_h(\W_h,\X_h;\V_h,-\X_h) \geq C_1 \|(\W_h, \X_h)\|_h^2 
\end{equation}
and
\begin{equation}
    \enorm{\V_h} \leq C_2 \|(\W_h, \X_h)\|_h.
\end{equation}
\end{thm}

\begin{proof}
Using Lemma~\ref{lem:trace}, we first obtain
\begin{align*}
    \Bf_h(\W_h,\X_h;\W_h,-\X_h)&=\enorm{\W_h}^2 - \alpha \sum_{E \in \Gh} h_E \| \Sig(\W_h)\N \|^2_{0,E}+\alpha \sum_{E \in \Gh} h_E \| \X_h \|_{0,E}^2 \\
    &\geq \left(1 - \frac{\alpha}{C_I}\right) \enorm{\W_h}^2 + \alpha \sum_{E \in \Gh} h_E \| \X_h \|_{0,E}^2.
\end{align*}
On the other hand, the continuous stability estimate implies that there exists $\V \in \bs{V}$ such that for any $\X_h \in \bs{Q}_h$ it holds
\begin{equation}
\label{eq:contstab1}
b(\V,\X_h) \geq C_1 (\| \xi_{h,n} \|_{-\frac12,\Gamma}^2 + \| \X_{h,t} \|_{-\frac12,\Gamma}^2)
\end{equation}
and, moreover, 
\begin{equation}
\label{eq:contstab2}
\enorm{\V}^2 \leq C_2 (\| \xi_{h,n} \|_{-\frac12,\Gamma}^2 + \| \X_{h,t} \|_{-\frac12,\Gamma}^2).
\end{equation}
Let $\widetilde{\V} \in \bs{V}_h$ now be the Cl\'ement interpolant of $\V$.  It follows that
\begin{equation}
\label{eq:verfurth}
\begin{aligned}
b(\widetilde{\V},\X_h) &= b( \widetilde{\V} - \V,\X_h) + b(\V,\X_h) \\
&\geq - \sum_{E \in \Gh} h_E^{-1/2} \| v_n - \widetilde{v}_n \|_{0,E}\, h_E^{1/2} \| \xi_{h,n} \|_{0,E} \\
&\quad- \sum_{E \in \Gh} h_E^{-1/2} \| \V_t - \widetilde{\V}_t \|_{0,E}\, h_E^{1/2} \| \X_{h,t} \|_{0,E} \\
&\quad +  C_1 (\| \xi_{h,n} \|_{-\frac12,\Gamma}^2 + \| \X_{h,t} \|_{-\frac12,\Gamma}^2)\\
&\geq - \left(\sum_{E \in \Gh} h_E^{-1} \| v_n - \widetilde{v}_n \|_{0,E}^2 \right)^{1/2} \left(\sum_{E \in \Gh} h_E \| \xi_{h,n} \|_{0,E}^2\right)^{1/2} \\
&\quad- \left(\sum_{E \in \Gh} h_E^{-1} \| \V_t - \widetilde{\V}_t \|_{0,E}^2 \right)^{1/2} \left( \sum_{E \in \Gh} h_E \| \X_{h,t} \|_{0,E}^2 \right)^{1/2} \\
&\quad +  C_1 (\| \xi_{h,n} \|_{-\frac12,\Gamma}^2 + \| \X_{h,t} \|_{-\frac12,\Gamma}^2) \\
&\geq - C_3 \enorm{\V} \left( \sum_{E \in \Gh} h_E \| \X_{h} \|_{0,E}^2 \right)^{1/2}  \hspace{-0.2cm}+  C_1 (\| \xi_{h,n} \|_{-\frac12,\Gamma}^2 + \| \X_{h,t} \|_{-\frac12,\Gamma}^2) \\
&\geq - C_4 \sum_{E \in \Gh} h_E \| \X_{h} \|_{0,E}^2 +  C_5 (\| \xi_{h,n} \|_{-\frac12,\Gamma}^2 + \| \X_{h,t} \|_{-\frac12,\Gamma}^2).
\end{aligned}
\end{equation}
where we have used Young's inequality,
estimates \eqref{eq:contstab1} and \eqref{eq:contstab2}, and the following properties
of the Cl\'ement interpolant:
\begin{equation}
\label{eq:clement}
\sum_{E \in \Gh} h_E^{-1} \| \V - \widetilde{\V} \|_{0,E}^2 \lesssim  \enorm{\V}^2 \quad \text{and} \quad \enorm{ \widetilde{\V}} \lesssim \enorm{\V}.
\end{equation}
Finally, combining the estimates above, we obtain the bound
\begin{align*}
    &\Bf_h(\W_h,\X_h;\W_h + \delta \widetilde{\V},-\X_h) \\
    &\geq  \left(1 - \frac{\alpha}{C_I}\right) \enorm{\W_h}^2 + \alpha \sum_{E \in \Gh} h_E \| \X_h \|_{0,E}^2 \\
    &\quad - \delta \enorm{\W_h} \enorm{\widetilde{\V}} + \delta b(\widetilde{\V},\X_h) - \alpha \delta \sum_{E \in \Gh} h_E(\X_h + \Sig(\W_h)\N, \Sig(\widetilde{\V})\N)_E \\
    &\gtrsim \| (\W_h, \X_h) \|_h^2,
\end{align*}
where the last step follows, choosing $\delta > 0$ small enough, from Young's inequality, estimates \eqref{eq:contstab2}, \eqref{eq:verfurth} and \eqref{eq:clement}, and Lemma~\ref{lem:trace}.
\end{proof}

\section{Error analysis}

Let $\F_h \in \bs{V}_h$ be the $[L^2(\Omega)]^d$-projection of $\F$. For any $K \in \Th$, we define the oscillation of $\F$ by $\mathrm{osc}_{K}(\F) = h_K\|\F - \F_h\|_{0,K}$. Moreover, we denote by $K(E) \in \Th$ the element having $E \in \Gh$ as one of its facets, i.e., $\partial K(E)\cap E = E$.

\begin{lem}
\label{lem:resbound}
For any $(\W_h, \X_h) \in \bs{V}_h \times \bs{\varLambda}_h$ it holds that
\begin{align*}
    &\left(\sum_{E \in \Gh} h_E \| \X_h + \Sig(\W_h)\N \|_{0,E}^2\right)^{1/2}\\
    &\qquad \lesssim \|(\U - \W_h, \Lam - \X_h) \| + \left(\sum_{E \in \Gh} \mathrm{osc}_{K(E)}(\F)^2 \right)^{1/2}.
\end{align*}
\end{lem}

\begin{proof}
Let $b_E : E \rightarrow [0, 1]$, $b_E \in P_d(E)$, denote the facet bubble with $b_E=1$ in the middle of $E$ and $b_E=0$ on the boundary $\partial E$. 
Then we have
\begin{equation}
\label{eq:bubble}
h_E \| \X_h + \Sig(\W_h)\N \|_{0,E}^2 \lesssim (\X_h + \Sig(\W_h)\N, \bs{\tau}_E)_E
\end{equation}
where
\[
\bs{\tau}_E|_E = b_E h_E ( \X_h + \Sig(\W_h)\N) \quad \text{and} \quad \bs{\tau}_E |_{\partial K(E) \setminus E} = \Z.
\]
Choosing $\V = \bs{\tau}$ in Problem~\ref{prob:weak} with $\bs{\tau} = \sum_{E \in \Gh} \bs{\tau}_E$ gives
\[
0 = -(\Sig(\U), \Eps(\bs{\tau})) - b(\bs{\tau}, \Lam) + (\F, \bs{\tau}).
\]
Summing \eqref{eq:bubble} over the edges $E \in \Gh$ and using the above equality gives
\begin{align*}
&\sum_{E \in \Gh} h_E \| \X_h + \Sig(\W_h)\N \|_{0,E}^2 \\ &\lesssim (\X_h + \Sig(\W_h)\N, \bs{\tau})_\Gamma\\
&= (\X_h, \bs{\tau})_\Gamma + (\Sig(\W_h)\N, \bs{\tau})_\Gamma -(\Sig(\U), \Eps(\bs{\tau})) - b( \bs{\tau}, \Lam) + (\F, \bs{\tau})  \\
&= b( \bs{\tau}, \X_h - \Lam) + (\Sig(\W_h)\N, \bs{\tau})_\Gamma -(\Sig(\U), \Eps(\bs{\tau})) + (\F, \bs{\tau})  \\
&= b( \bs{\tau}, \X_h - \Lam ) + (\Sig(\W_h - \U), \Eps(\bs{\tau})) + (\Div \Sig(\W_h) + \F, \bs{\tau}). 
\end{align*}
We conclude the proof by estimating the terms on the right-hand side using the Cauchy--Schwarz and trace inequalities, the inverse estimate
\[
\|\bs{\tau}\|_{1}^2 \lesssim \sum_{E \in \Gh} h_E^{-2} \| \bs{\tau}_E \|_{0,E}^2,
\]
and the standard estimates for interior residuals~\cite{VerfurthBook} which give rise to the oscillation term.
\end{proof}

Based on Lemma \ref{lem:resbound}, we will now prove the quasi-optimality of the method. We are deliberately not touching the subject of a priori bounds since we wish to avoid making assumptions on the regularity of the solution (see also Remark~4.1 in \cite{GSV2020}).
\begin{thm}[quasi-optimality]
Suppose that $0 < \alpha < C_I$.
For any $(\W_h, \X_h) \in \bs{V}_h \times \bs{\varLambda}_h$ it holds
\begin{align*}
\|(\U - \U_h, \Lam - \Lam_h)\| &\lesssim \|(\U - \W_h, \Lam - \X_h) \| + \sqrt{( u_n - g, \xi_{h,n})_\Gamma} \\
& \quad + \sqrt{\langle\U_t, \Lam_t - \X_{h,t}\rangle} + \left(\sum_{E \in \Gh} \mathrm{osc}_{K(E)}(\F)^2 \right)^{1/2}.
\end{align*}
\end{thm}

\begin{proof}
    Let $\V_h \in \bs{V}_h$ be, according to Theorem~\ref{thm:disc}, such that
    \begin{equation}
        \label{eq:discstabthm}
        \enorm{\V_h} \lesssim \|(\U_h - \W_h, \Lam_h - \X_h)\|_h.
    \end{equation}
    From the discrete problem \ref{discprob} it follows that
    \[
        \Bf_h(\U_h,\Lam_h;\V_h,\X_h-\Lam_h) \leq \Lf(\V_h,\X_h - \Lam_h)\quad \forall \X_h\in \bs{\varLambda}_h,
    \]
    and, consequently,
    \begin{align*}
        &\|(\U_h - \W_h, \Lam_h - \X_h)\|^2_h \\
        &\lesssim \Bf_h(\U_h - \W_h, \Lam_h - \X_h; \V_h, \X_h - \Lam_h) \\
        &\leq \Lf(\V_h, \X_h - \Lam_h) - \Bf_h(\W_h, \X_h; \V_h, \X_h - \Lam_h) \\
        &=\Lf(\V_h, \X_h - \Lam_h) + \Bf(\U - \W_h, \Lam - \X_h; \V_h, \X_h - \Lam_h) \\
        &\phantom{=} - \Bf(\U,\Lam;\V_h,\X_h-\Lam_h) + \alpha \sum_{E \in \Gh} h_E(\X_h + \Sig(\W_h)\N, \X_h - \Lam_h + \Sig(\V_h)\N)_E.
    \end{align*}
    The second term above is bounded using the continuity of $\Bf$.
    The first and the third terms are simplified using the
    continuous problem as follows:
    \begin{align*}
        &\Lf(\V_h, \X_h - \Lam_h) - \Bf(\U, \Lam; \V_h, \X_h - \Lam_h) \\
        &=(\F,\V_h)-(\Sig(\U),\Eps(\V_h))-b( \V_h, \Lam ) + \langle g,\xi_{h,n} - \lambda_{h,n} \rangle -b( \U, \X_h - \Lam_h ) \\
        &=\langle g,\xi_{h,n} - \lambda_{h,n} \rangle -b( \U, \X_h - \Lam_h ) \\
        &=(g-u_n, \xi_{h,n} - \lambda_{h,n})_\Gamma + \langle \U_t, \Lam_{h,t} - \X_{h,t}\rangle\\
        &\leq (g-u_n, \xi_{h,n} )_\Gamma + (\U_t, \Lam_t - \X_{h,t})_\Gamma,
    \end{align*}
    where the last inequality follows from the fact that we consider a conforming approximation, i.e.~$\lambda_{h,n} \geq 0$ and
    \[
    \langle \U_t, \Lam_{h,t} - \X_{h,t} \rangle = \langle \U_t, \Lam_t - \X_{h,t} \rangle + \langle \U_t, \Lam_{h,t} - \Lam_t \rangle \leq \langle \U_t, \Lam_t - \X_{h,t} \rangle.
    \]
    Finally, we bound the terms due to stabilization as follows:
    \begin{align*}
    &\sum_{E \in \Gh} h_E(\X_h + \Sig(\W_h)\N, \X_h - \Lam_h + \Sig(\V_h)\N)_E\\
    &\lesssim \left(\sum_{E \in \Gh} h_E \|\X_h + \Sig(\W_h)\N\|_{0,E}^2 \right)^{1/2} \left(\sum_{E \in \Gh} h_E \|\X_h - \Lam_h + \Sig(\V_h)\N\|_{0,E}^2 \right)^{1/2} \\
        &\lesssim \left(\sum_{E \in \Gh} h_E \|\X_h + \Sig(\W_h)\N\|_{0,E}^2 \right)^{1/2}\\ &\qquad \cdot \left(\sum_{E \in \Gh} h_E \|\X_h - \Lam_h\|_{0,E}^2 + \sum_{E \in \Gh} h_E \| \Sig(\V_h)\N\|_{0,E}^2 \right)^{1/2} \\
        &\lesssim \left(\sum_{E \in \Gh} h_E \|\X_h + \Sig(\W_h)\N\|_{0,E}^2 \right)^{1/2} \|(\U_h - \W_h, \Lam_h - \X_h)\|_h
    \end{align*}
    where we have used the Cauchy--Schwarz and the triangle inequalities, estimate \eqref{eq:discstabthm} and Lemma \ref{lem:trace}.
    Now, the result follows taking into account Lemma~\ref{lem:resbound} and the trivial bound
    \[
    \|(\U_h - \W_h, \Lam_h - \X_h)\| \leq \|(\U_h - \W_h, \Lam_h - \X_h)\|_h,
    \]
    and using the triangle inequality.
\end{proof}

Next we define the total error indicator
\[
\eta^2 = \sum_{K \in \Th} \eta_K^2 + \sum_{E \in \Eh} \eta_{E, \Omega}^2 + \sum_{E \in \Fh} \eta_{E,\Gamma_N}^2 + \sum_{E \in \Gh} \eta_{E, \Gamma}^2
\]
where the local indicators are given by
\begin{alignat*}{2}
\eta_K^2 &= h_K^2 \| \Div \Sig(\U_h) + \F \|_{0,K}^2, \quad &&K \in \Th,\\
\eta_{E,\Omega}^2 &= h_E \| \llbracket \Sig(\U_h) \N \rrbracket \|_{0,E}^2, \quad &&E \in \Eh,\\
\eta_{E,\Gamma_N}^2 &= h_E \| \Sig(\U_h)\N \|_{0,E}^2, \quad &&E \in \Fh,\\
\eta_{E,\Gamma}^2 &= h_E \| \Lam_h + \Sig(\U_h)\N \|_{0,E}^2, \quad &&E \in \Gh.
\end{alignat*}
Moreover, we define the additional terms at the contact boundary as
\[
S^2 =  \| (g - u_{h,n})_-\|_{0,\Gamma}^2+((g - u_{h,n})_+,\lambda_{h,n})_\Gamma + \int_\Gamma (\kappa | \U_{h,t} | - \U_{h,t} \cdot \Lam_{h,t})\,\mathrm{d}s.
\]

\begin{thm}[a posteriori error estimate]
    \[
    \|(\U-\U_h, \Lam-\Lam_h)\| \lesssim \eta + S
    \]
\end{thm}

\begin{proof}
In view of continuous stability there exists $\V \in \bs{V}$ such that
\begin{align*}
    &\| (\U - \U_h, \Lam - \Lam_h) \|^2\\
    &\lesssim \Bf(\U - \U_h, \Lam - \Lam_h; \V, \Lam_h - \Lam)\\
    &=\Bf(\U, \Lam, \V, \Lam_h - \Lam) - \Bf(\U_h, \Lam_h; \V, \Lam_h - \Lam) \\
    &\leq \Lf(\V, \Lam_h - \Lam) - \Bf(\U_h, \Lam_h; \V, \Lam_h - \Lam),
\end{align*}
where in the last inequality we have used Problem \ref{prob:weak}.
Let now $\widetilde{\V}$ be the Cl\'ement interpolant of $\V$. From the discrete problem \eqref{discprob}, it follows that
\[
0 \leq - \Bf(\U_h, \Lam_h; -\widetilde{\V}, \Z) + \Lf(-\widetilde{\V}, \Z) + \alpha \sum_{E \in \Gh} h_E (\Lam_h + \Sig(\U_h)\N, \Sig(-\widetilde{\V})\N)_E.
\]
Using the above inequality and integration by parts, we get
\begin{align*}
    &\| (\U - \U_h, \Lam - \Lam_h) \|^2\\
    &\lesssim \Lf(\V - \widetilde{\V}, \Lam_h - \Lam) - \Bf(\U_h, \Lam_h; \V - \widetilde{\V}, \Lam_h - \Lam) \\
    &\phantom{=}~+ \alpha \sum_{E \in \Gh} h_E (\Lam_h + \Sig(\U_h)\N, \Sig(-\widetilde{\V})\N)_E \\
    &=\sum_{K \in \Th}( \Div \Sig(\U_h) + \F, \V - \widetilde{\V})_K - \sum_{E \in \Eh} (\llbracket \Sig(\U_h)\N \rrbracket, \V - \widetilde{\V})_E \\
    &\phantom{=}~- \sum_{E \in \Fh} ( \Sig(\U_h)\N , \V - \widetilde{\V})_E - \sum_{E \in \Gh}(\Lam_h + \Sig(\U_h)\N, \V - \widetilde{\V})_E \\ 
    &\phantom{=}~- b( \U_h, \Lam_h - \Lam) +\langle g, \lambda_{h,n} - \lambda_n \rangle + \alpha \sum_{E \in \Gh} h_E (\Lam_h + \Sig(\U_h)\N, \Sig(-\widetilde{\V})\N)_E.
\end{align*}
The first four terms are bounded using the Cauchy--Schwarz inequality, continuous stability estimate,
and the standard Clem\'ent interpolation estimate:
\[
\sum_{K \in \Th} h_K^{-2} \| \V - \widetilde{\V} \|_{0,K}^2 + \sum_{E \in \Eh \cup \Fh \cup \Gh} h_E^{-1} \| \V - \widetilde{\V} \|_{0,E}^2 \lesssim  \enorm{\V}^2.
\]
The last term is estimated using the Cauchy--Schwarz inequality, Lemma~\ref{lem:trace},
and the bound $\enorm{\widetilde{\V}} \lesssim \enorm{\V}$.
The remaining terms are estimated as follows:
\begin{align*}
   & -b(\U_h, \Lam_h - \Lam) +\langle g, \lambda_{h,n} - \lambda_n \rangle \\
   &= \langle g - u_{h,n}, \lambda_{h,n} - \lambda_n \rangle - \langle \U_{h,t}, \Lam_{h,t} - \Lam_t \rangle \\
   &= \langle (g - u_{h,n})_-, \lambda_{h,n} - \lambda_n \rangle + \langle (g - u_{h,n})_+, \lambda_{h,n} - \lambda_n \rangle \\
   &\phantom{=}~- ( \U_{h,t}, \Lam_{h,t} )_\Gamma + \langle \U_{h,t}, \Lam_t \rangle \\
   &\leq \| (g - u_{h,n})_- \|_{1/2,\Gamma} \|\lambda_{h,n} - \lambda_n\|_{-1/2,\Gamma} + ((g - u_{h,n})_+,\lambda_{h,n})_\Gamma \\
   &\phantom{\leq}~+ \int_\Gamma (\kappa | \U_{h,t} | - \U_{h,t} \cdot \Lam_{h,t})\,\mathrm{d}s.
\end{align*}

\end{proof}

The proof of the following theorem is standard~\cite{VerfurthBook}; see also \cite{GSV2017}.
\begin{thm}[efficiency of the error indicator]
\[
\eta \lesssim \|(\U - \U_h, \Lam - \Lam_h)\| + \left(\sum_{E \in \Gh} \mathrm{osc}_{K(E)}(\F)^2 \right)^{1/2}.
\]
\end{thm}

\section{Nitsche's method}

Due to the stabilization terms, the Lagrange multiplier can be
eliminated locally in each element.
Choosing $\V_h = \Z$ in the discrete problem gives
\begin{align*}
&(\U_h, \Mu_h - \Lam_h)_\Gamma - \alpha \sum_{E \in \Gh} h_E (\Lam_h + \Sig(\U_h)\N, \Mu_h - \Lam_h)_E\\
&\leq (g, \mu_{h,n} - \lambda_{h,n})_\Gamma \quad \forall \Mu_h \in \bs{\varLambda}_h.
\end{align*}
This inequality can be  decomposed into normal and tangential parts by considering the test function $\Mu_h = \mu_{h,n}\N + \Mu_{h,t}$. Choosing first $\Mu_{h,t} = \Lam_{h,t}$ and then $\mu_{h,n} = \lambda_{h,n}$ leads to the inequalities
\[
(u_{h,n} - g, \mu_{h,n} - \lambda_{h,n})_\Gamma - \alpha \sum_{E \in \Gh} h_E (\lambda_{h,n} + \sigma_n(\U_h), \mu_{h,n} - \lambda_{h,n})_E \leq 0
\]
and
\[
(\U_{h,t},\Mu_{h,t} - \Lam_{h,t})_\Gamma - \alpha \sum_{E \in \Gh} h_E (\Lam_{h,t} + \Sig_t(\U_h)\N, \Mu_{h,t} - \Lam_{h,t})_E \leq 0,
\]
valid for every $(\mu_{h,n}, \Mu_{h,t}) \in \bs{\varLambda}_h$.
Suppose now that $E \in \Gh$ is such that $|\Lam_{h,t}(\bs{x})| < \kappa$ for any $\bs{x} \in E$.  Then the test function
\[
\Mu_{h,t}(\bs{x}) = \begin{cases}
    \Lam_{h,t}(\bs{x}) \pm \varepsilon \bs{\varphi}_E(\bs{x}) & \text{if $\bs{x} \in E$,} \\
    \Lam_{h,t}(\bs{x}) & \text{otherwise,}
\end{cases}
\]
where $\bs{\varphi}_E$ is one of the basis functions of $\bs{Q}_h|_E$ and $\varepsilon > 0$ is sufficiently
small so that $|\Mu_{h,t}|<\kappa$,
gives
\[
\begin{cases}
(\U_{h,t} - \alpha h_E(\Lam_{h,t} + \Sig_t(\U_h)\N), \bs{\varphi}_E)_E \leq 0\\ 
(\U_{h,t} - \alpha h_E(\Lam_{h,t} + \Sig_t(\U_h)\N), -\bs{\varphi}_E)_E \leq 0,
\end{cases}
\]
implying that
\[
(\U_{h,t} - \alpha h_E(\Lam_{h,t} + \Sig_t(\U_h)\N), \bs{\varphi}_E)_E = 0.
\]
Thus, we obtain the expression
\[
\Lam_{h,t} = \bs{\gamma}_t(\U_h) := \frac{1}{\alpha \mathcal{H}}\pi_h \U_{h,t} - \pi_h \Sig_t(\U_h),
\]
where $\pi_h$ denotes the $L^2$-projection onto $\bs{Q}_h$ and $\mathcal{H} : \Gamma \rightarrow \mathbb{R}$, is such that $\mathcal{H}|_E = h_E$.
Taking into account the entire boundary $\Gamma$ and the case $|\Lam_{h,t}| = \kappa$, the expression for the tangential Lagrange multiplier then reads as follows
\begin{equation}
\label{eq:tanglag}
\Lam_{h,t} = \begin{cases}
    \bs{\gamma}_t(\U_h) & \text{if $|\bs{\gamma}_t(\U_h)|<\kappa$,} \\
    \kappa \frac{\bs{\gamma}_t(\U_h)}{|\bs{\gamma}_t(\U_h)|} & \text{otherwise.}
\end{cases}
\end{equation}
Taking similar steps to
eliminate the normal-directional Lagrange multiplier,
we arrive at the expression
\begin{equation}
\label{eq:normlag}
\lambda_{h,n} = \begin{cases}
    \gamma_n(\U_h) & \text{if $\gamma_n(\U_h) > 0$,} \\
    0 & \text{otherwise,}
\end{cases}
\end{equation}
where
\[
\gamma_n(\U_h) := \frac{1}{\alpha \mathcal{H}}( \pi_h u_{h,n} - \pi_h g) - \pi_h \sigma_n(\U_h).
\]

 Testing with $\V_h$ in the discrete problem \eqref{discprob} and using  expressions \eqref{eq:tanglag} and \eqref{eq:normlag}, leads to the variational equality
\begin{equation}
\label{eq:first}
\begin{aligned}
&(\Sig(\U_h), \Eps(\V_h))_\Omega \\
&+ b(\V_h, \Lam_h) - \alpha \sum_{E \in \Gh} h_E(\Lam_h + \Sig(\U_h)\N, \Sig(\V_h)\N)_E \\
&= (\F, \V_h)_\Omega \quad \forall \V_h \in \bs{V}_h.
\end{aligned}
\end{equation}
Focusing on the terms
\begin{equation}
\label{eq:nitscheterms}
b(\V_h, \Lam_h) - \alpha \sum_{E \in \Gh} h_E(\Lam_h + \Sig(\U_h)\N, \Sig(\V_h)\N)_E,
\end{equation}
we first observe that the normal-directional part of \eqref{eq:nitscheterms} reads as follows
\[
(\lambda_{h,n}, v_{h,n})_\Gamma - \alpha \sum_{E \in \Gh} h_E(\lambda_{h,n} + \sigma_n(\U_h), \sigma_n(\V_h))_E .
\]
In view of \eqref{eq:normlag} and the definition of $\gamma_n(\U_h)$, this can be written as
\begin{equation}
\label{eq:nitschenorm}
\begin{aligned}
&\left(\tfrac{1}{\alpha \mathcal{H}} u_{h,n}, v_{h,n} \right)_{\Gamma_C} - (\alpha \mathcal{H} \sigma_n(\U_h),\sigma_n(\V_h))_{\Gamma \setminus \Gamma_C}\\
&- (\sigma_n(\U_h), v_{h,n})_{\Gamma_C} - (u_{h,n},\sigma_n(\V_h))_{\Gamma_C} \\
&-\left(\tfrac{1}{\alpha \mathcal{H}} \pi_h g, v_{h,n}\right)_{\Gamma_C} +(\alpha \mathcal{H} \pi_h g, \sigma_n(\V_h))_{\Gamma_C}
\end{aligned}
\end{equation}
provided that $\bs{Q}_h$ has a large enough polynomial order so that $\pi_h \sigma_n(\U_h) = \sigma_n(\U_h)$ and $\pi_h u_{h,n} = u_{h,n}$.
Above, and in what follows,
\[
\Gamma_C = \Gamma_C(\U_h) = \{ \bs{x} \in \Gamma : \gamma_n(\U_h(\bs{x})) > 0 \}
\]
is defined as the active contact boundary.

The tangential part of \eqref{eq:nitscheterms} is
\[
(\Lam_{h,t}, \V_{h,t})_\Gamma - \alpha \sum_{E \in \Gh} h_E(\Lam_{h,t} + \Sig_t(\U_h), \Sig_t(\V_h))_E
\]
and, using \eqref{eq:tanglag}, it can  be expanded into
\begin{equation}
\label{eq:nitschetang}
\begin{aligned}
&\left(\tfrac{1}{\alpha \mathcal{H}} \U_{h,t}, \V_{h,t} \right)_{\Gamma_S} - (\alpha \mathcal{H} \Sig_t(\U_h),\Sig_t(\V_h))_{\Gamma \setminus \Gamma_S}\\
&- (\Sig_t(\U_h), \V_{h,t})_{\Gamma_S} - (\U_{h,t},\Sig_t(\V_h))_{\Gamma_S} \\
&+\left(\kappa \tfrac{\bs{\gamma}_t(\U_h)}{|\bs{\gamma}_t(\U_h)|}, \V_{h,t}\right)_{\Gamma \setminus \Gamma_S}-\left(\alpha \mathcal{H} \kappa \tfrac{\bs{\gamma}_t(\U_h)}{|\bs{\gamma}_t(\U_h)|}, \Sig_t(\V_h)\right)_{\Gamma \setminus \Gamma_S}
\end{aligned}
\end{equation}
where we have defined $\Gamma_S = \Gamma_S(\U_h) = \{ \bs{x} \in \Gamma : |\bs{\gamma}_t(\U_h(\bs{x}))| < \kappa \}$ as the part of $\Gamma$ where
the body sticks to the rigid foundation.

Combining \eqref{eq:first}, \eqref{eq:nitschenorm} and \eqref{eq:nitschetang} leads to a Nitsche-type formulation of the discrete problem
wherein the Lagrange multiplier is absent.  Note that the new formulation is still nonlinear
because the sets $\Gamma_C$ and $\Gamma_S$ depend on the solution $\U_h$ itself.
Inspired by the primal-dual active set strategy proposed in~\cite{W11}, i.e.~fixing the active sets
$\Gamma_C$ and $\Gamma_S$ by repeatedly limiting the  components of $\Lam_h$ that violate the contact/friction conditions,
we perform a fixed-point iteration where
$\Gamma_C$ and $\Gamma_S$ are computed from the discrete solution of the previous iteration.
The resulting method is summarized in Algorithm~\ref{alg:nitsche}.

\begin{algorithm}[H]
  \caption{Contact iteration}
  \label{alg:nitsche}
  \begin{algorithmic}[1]
    \Require{$\bs{V}_h$ is a finite element space}
    \Require{$\epsilon > 0$ is a convergence tolerance}
    \Statex
    \Procedure{Nitsche}{$\bs{V}_h$, $\epsilon$}
    \State $\W_h \gets \Z$
    \Repeat
        \State Find $\U_h \in \bs{V}_h$ such that
        \begin{align*}
            &(\Sig(\U_h), \Eps(\V_h))_\Omega \\
            &+\left(\tfrac{1}{\alpha \mathcal{H}} u_{h,n}, v_{h,n} \right)_{\Gamma_C(\W_h)} - (\alpha \mathcal{H} \sigma_n(\U_h),\sigma_n(\V_h))_{\Gamma \setminus \Gamma_C(\W_h)}\\
&- (\sigma_n(\U_h), v_{h,n})_{\Gamma_C(\W_h)} - (u_{h,n},\sigma_n(\V_h))_{\Gamma_C(\W_h)} \\
&+\left(\tfrac{1}{\alpha \mathcal{H}} \U_{h,t}, \V_{h,t} \right)_{\Gamma_S(\W_h)} - (\alpha \mathcal{H} \Sig_t(\U_h),\Sig_t(\V_h))_{\Gamma \setminus \Gamma_S(\W_h)}\\
&- (\Sig_t(\U_h), \V_{h,t})_{\Gamma_S(\W_h)} - (\U_{h,t},\Sig_t(\V_h))_{\Gamma_S(\W_h)} \\
&=(\F, \V_h)_\Omega \\
&\quad -\left(\tfrac{1}{\alpha \mathcal{H}} \pi_h g, v_{h,n}\right)_{\Gamma_C(\W_h)} +(\alpha \mathcal{H} \pi_h g, \sigma_n(\V_h))_{\Gamma_C(\W_h)}\\
&\quad +\left(\kappa \tfrac{\bs{\gamma}_t(\W_h)}{|\bs{\gamma}_t(\W_h)|}, \V_{h,t}\right)_{\Gamma \setminus \Gamma_S(\W_h)}\\
&\quad-\left(\alpha \mathcal{H} \kappa \tfrac{\bs{\gamma}_t(\W_h)}{|\bs{\gamma}_t(\W_h)|},\Sig_t(\V_h)\right)_{\Gamma \setminus \Gamma_S(\W_h)} \quad \forall \V_h \in \bs{V}_h
        \end{align*}
        \State $\W_h \gets \U_h$
    \Until{$\enorm{\U_h - \W_h} < \epsilon$}
    \State \Return{$\U_h$}
    \EndProcedure
  \end{algorithmic}
\end{algorithm}

\begin{rem}
\label{rem:inexact}
There are several possibilities for performing the comparisons \mbox{$\gamma_n(\U_h) > 0$} and $|\bs{\gamma}_t(\U_h)| < \kappa$
in a practical implementation of Algorithm~\ref{alg:nitsche}.
For example, the sign of $\gamma_n(\U_h)$ may change inside an element and, therefore,
an exact integration is unfeasible unless the active contact boundary $\Gamma_C$ is known a priori.
Some options include defining the active contact boundary element-by-element while looking
at the mean value of $\gamma_n(\U_h)$, or performing the comparison $\gamma_n(\U_h) > 0$
separately at each quadrature point.
\end{rem}

\begin{rem}
The stabilized method of Problem~\ref{prob:stab} can be implemented directly via
a primal-dual active set strategy as demonstrated in \cite{GSV2017} for the related obstacle problem.
Such an approach is straightforward to realize using a piecewise-constant or discontinuous piecewise-linear
Lagrange multiplier
and may be computationally less intensive than Algorithm~\ref{alg:nitsche} since it avoids the reassembly of the
terms \eqref{eq:nitschenorm} and \eqref{eq:nitschetang} during each iteration.
\end{rem}

\section{Numerical experiment}

Let us consider the domain $\Omega = (-0.5, 0.5)^2$ with $x=0.5$, $x=-0.5$ and $y=\pm 0.5$
corresponding to $\Gamma$, $\Gamma_D$ and $\Gamma_N$, respectively.
We set $\F=\Z$, $g=-0.1$, $E=1$, $\nu=0.3$, and $\kappa=0.2$.
The numerical experiment is implemented using scikit-fem~\cite{skfem2020}
which relies heavily on the SciPy ecosystem~\cite{scipy}.
The figures are created with the help of matplotlib~\cite{hunter2007}
and the full source code of the experiment is available in~\cite{zenodo}.

Using quadratic finite elements, a uniform mesh with $h \approx 0.044$ and $\alpha=10^{-3}$ for Algorithm~\ref{alg:nitsche}, we obtain 
the Lagrange multipliers depicted in Figure~\ref{fig:ex1}.  The absolute value of the tangential multiplier is limited by the bound $\kappa=0.2$, as expected,
and the normal multiplier remains positive on the entire contact boundary.
In the absence of an analytical solution, we evaluate the vectorial $H^1$-norm of the discrete displacement and the total error estimator $\eta$
for a sequence of uniform meshes; see Table~\ref{tab:ex1}.

In an attempt to improve over the uniform meshing strategy,
we solve the same problem using adaptive mesh refinement
and terminate the refinement loop after the number of
degrees-of-freedom $N$ is above a given threshold.
The results calculated using the final adaptive mesh are given in Figure~\ref{fig:ex1adapt}.
Visually, the improvement in the Lagrange multipliers is obvious although $N$ is roughly
the same as in Figure~\ref{fig:ex1}.
In addition, comparing the values of $\|\U_h\|_1$ in Table~\ref{tab:ex1} and \ref{tab:ex1adapt} reveal that the vectorial
$H^1$-norm of the solution is close to what would be expected
from the uniform mesh with 132,098 degrees-of-freedom.

A careful comparison of the values of $\eta$ in Table~\ref{tab:ex1} and \ref{tab:ex1adapt} reveal
that asymptotically the adaptive meshing appears to improve the convergence rate
to $O(N^{-1})$ which is the expected rate of convergence for a smooth solution and quadratic elements;
see also Figure~\ref{fig:ex1conv}.
Thus, we conclude that while the convergence rate of
the uniform strategy is limited by the regularity of the exact solution,
the adaptive strategy successfully regains the optimal convergence rate
with respect to the number of degrees-of-freedom.

\begin{figure}[h!]
    \centering
    \includegraphics[width=0.4\textwidth]{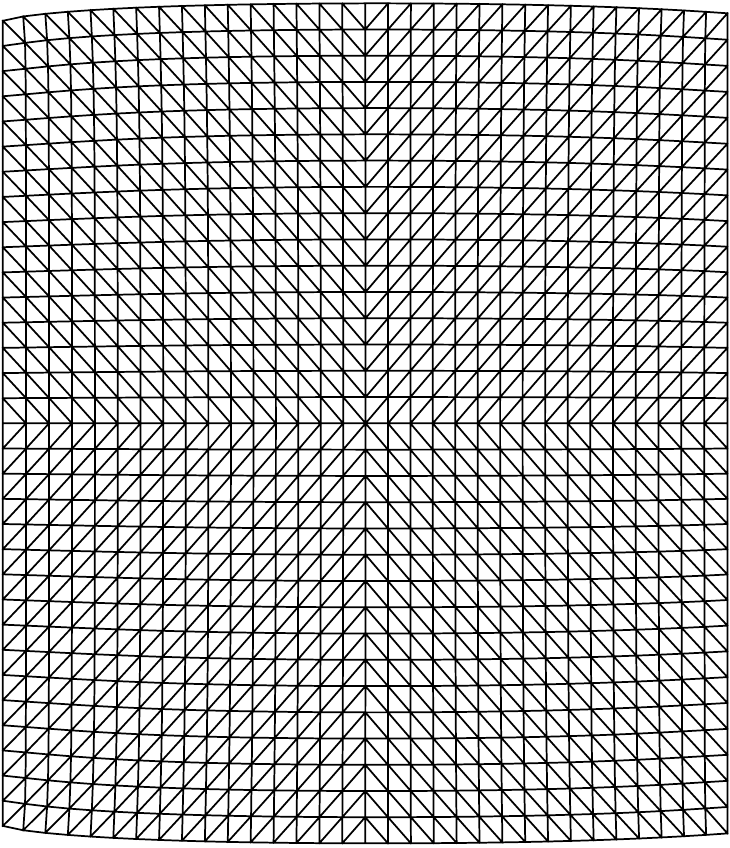}\\[0.5cm]
    \includegraphics[width=0.6\textwidth]{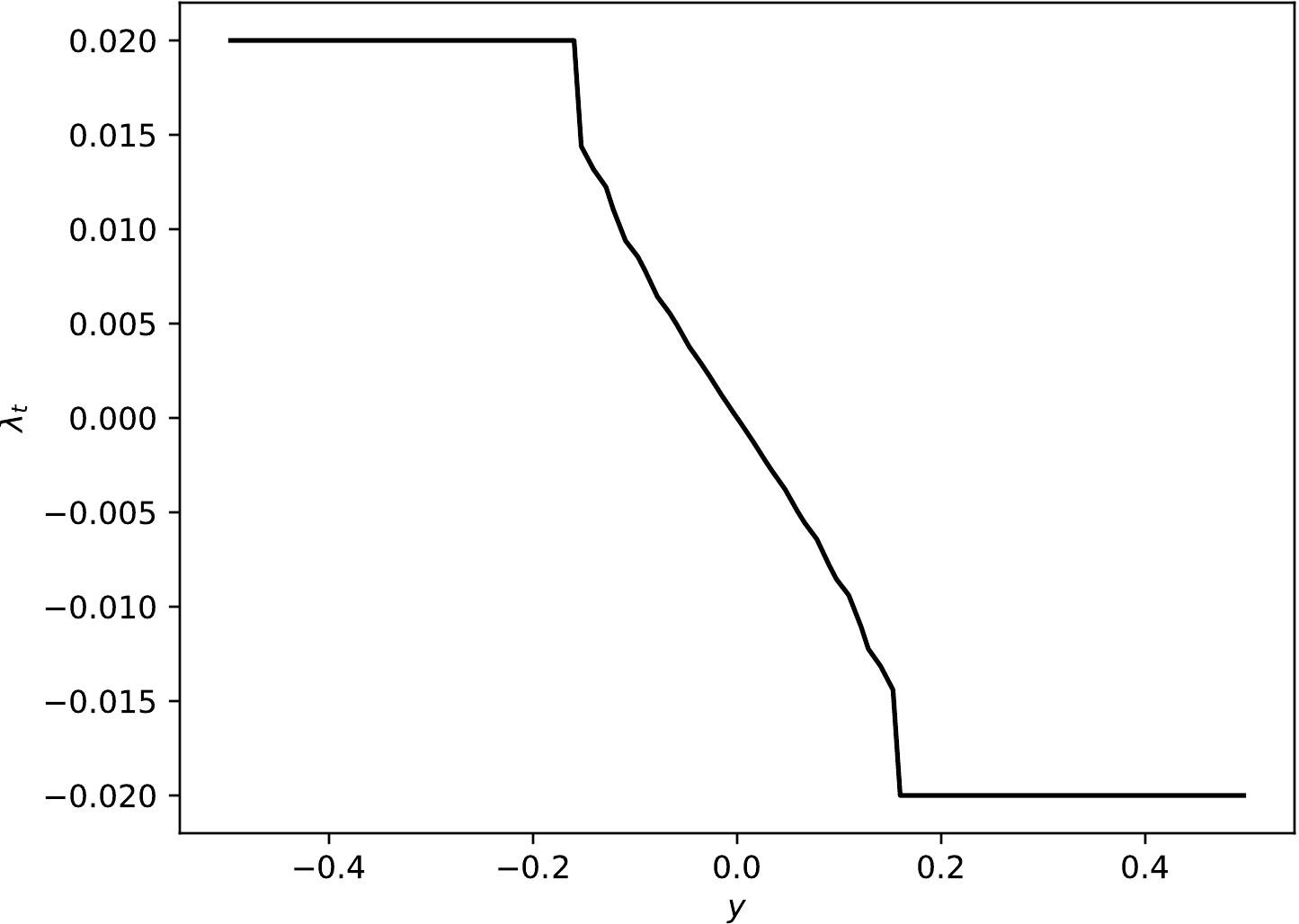}\\[0.5cm]
    \includegraphics[width=0.6\textwidth]{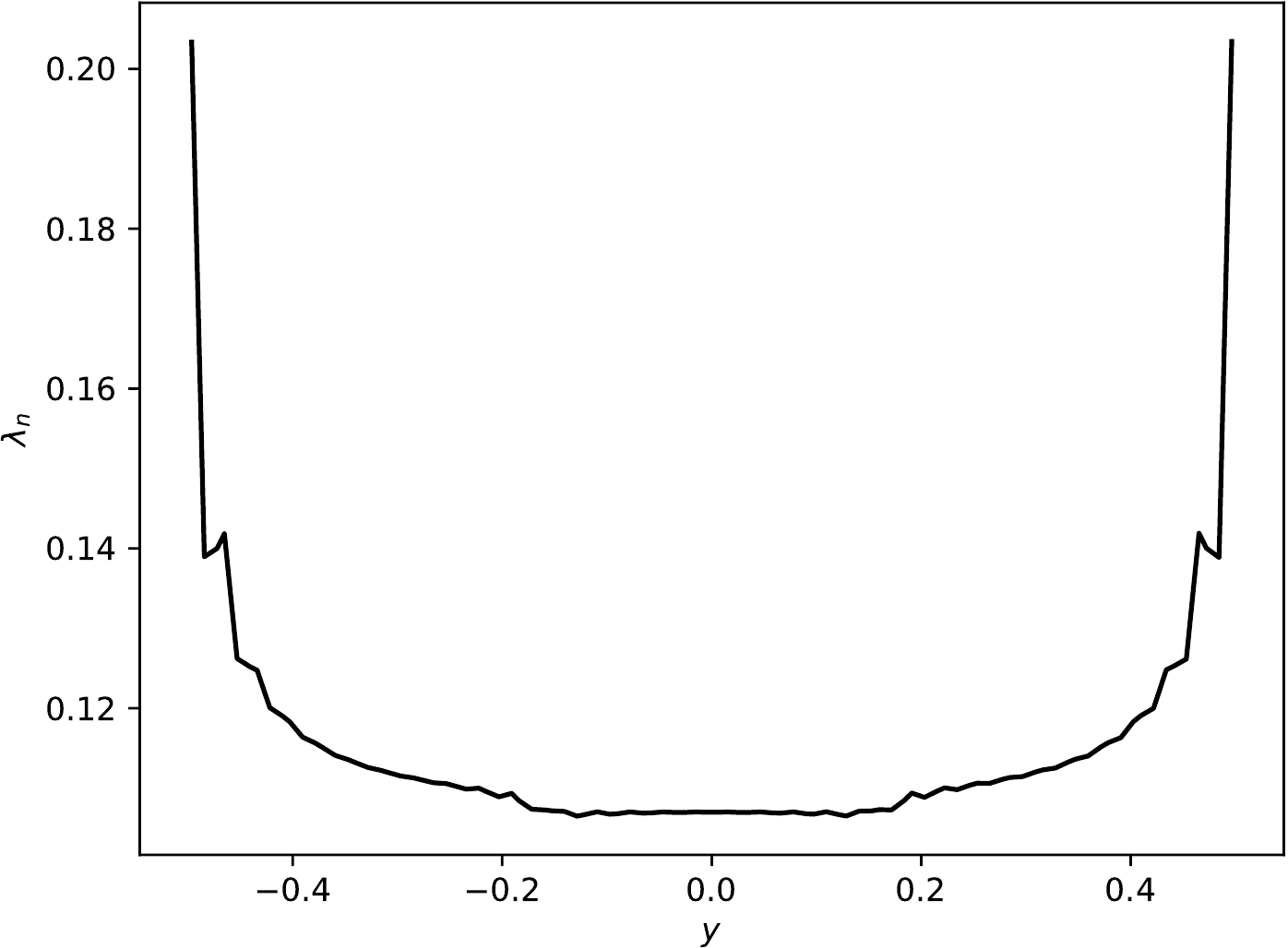}
    \caption{From top to bottom: the deformed, uniformly refined mesh with the number of degrees-of-freedom $N=8{,}450$, the tangential and the normal Lagrange multiplier as a function of $y$.}
    \label{fig:ex1}
\end{figure}

\begin{table}[h!]
  \caption{The norm of the solution and the value of the total error indicator $\eta$ using a uniform mesh family.
  The number of degrees-of-freedom is denoted by $N$.}
  \label{tab:ex1}
  \centering
  \pgfplotstabletypeset[
    create on use/rate/.style={
      create col/gradient loglog={N}{eta},
    },
    clear infinite,
    col sep=comma,
    sci zerofill={true},
    precision=2,
    columns={h,N,norm,eta},
    columns/h/.style={
     column name={$h$},
    },
    columns/N/.style={
     column name={$N$},
     fixed,
    },
    columns/norm/.style={
     column name={$\|\U_h\|_1$},
     precision=6,
    },
    columns/eta/.style={
     column name={$\eta$},
    },
  ]{
h,N,norm,eta
0.3535533905932738,162,0.12512491088285752,0.024313763514359765
0.1767766952966369,578,0.12521228022856246,0.01433158681806633
0.08838834764831845,2178,0.12533660448538167,0.008507952881306404
0.04419417382415922,8450,0.12536196044032774,0.00505894403542394
0.02209708691207961,33282,0.12537688747083747,0.003033564404895748
0.011048543456039806,132098,0.12538238166705057,0.0018265267263056603
}
\end{table}

\begin{figure}[h!]
    \centering
    \includegraphics[width=0.4\textwidth]{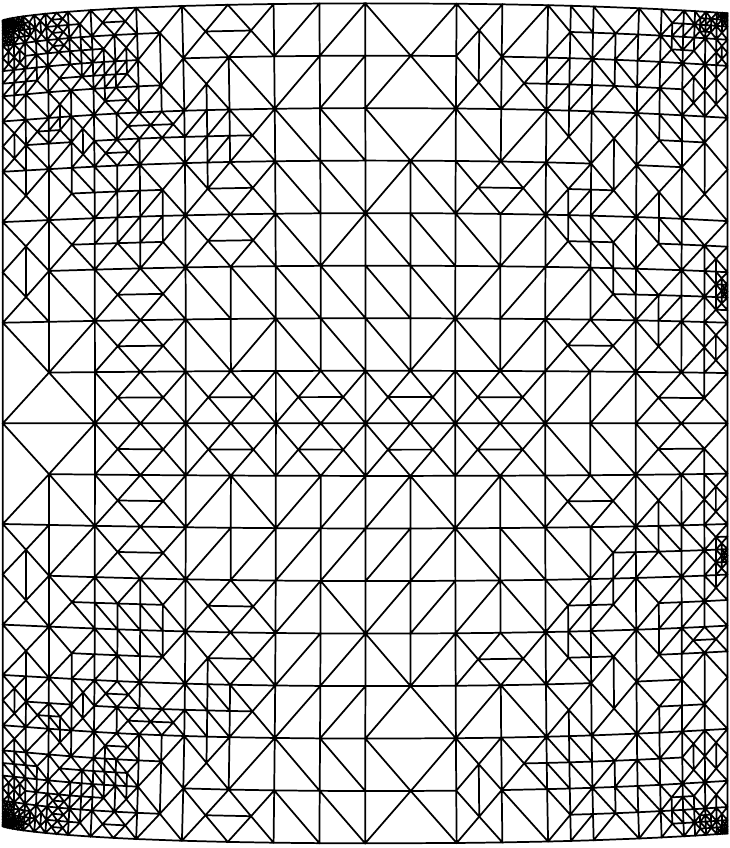}\\[0.5cm]
    \includegraphics[width=0.6\textwidth]{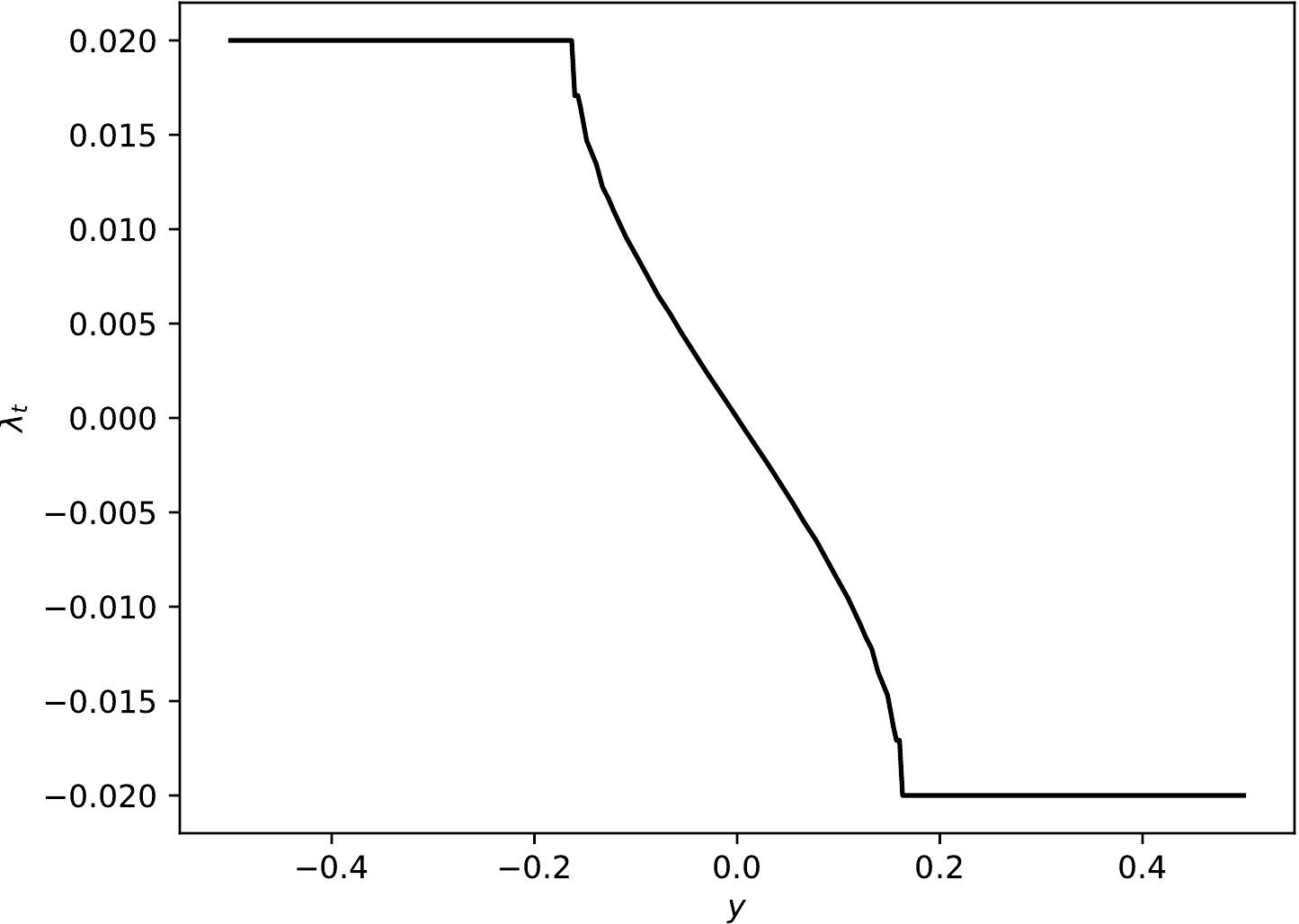}\\[0.5cm]
    \includegraphics[width=0.6\textwidth]{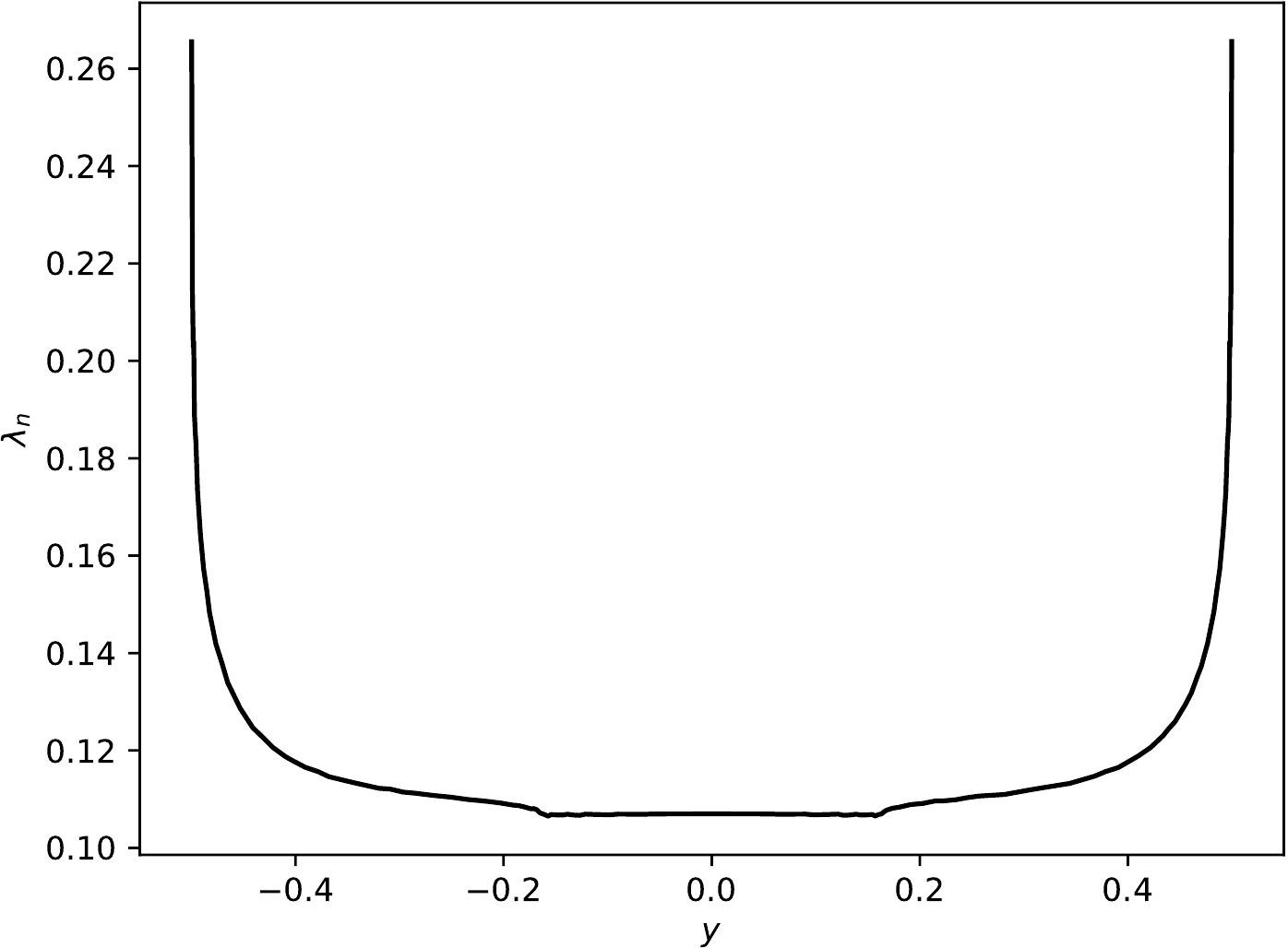}
    \caption{From top to bottom: the deformed, adaptively refined mesh with the number of degrees-of-freedom $N=7{,}946$, the tangential and the normal Lagrange multiplier as a function of $y$.}
    \label{fig:ex1adapt}
\end{figure}

\begin{table}[h!]
  \caption{The norm of the solution and the value of the total error indicator $\eta$ using an adaptive mesh family.
  The number of degrees-of-freedom is denoted by $N$.}
  \label{tab:ex1adapt}
  \centering
  \pgfplotstabletypeset[
    create on use/rate/.style={
      create col/gradient loglog={N}{eta},
    },
    clear infinite,
    col sep=comma,
    sci zerofill={true},
    precision=2,
    columns={N,norm,eta},
    columns/N/.style={
     column name={$N$},
     fixed,
    },
    columns/norm/.style={
     column name={$\|\U_h\|_1$},
     precision=6,
    },
    columns/eta/.style={
     column name={$\eta$},
    },
  ]{
N,norm,eta
162,0.12512491088285752,0.024313763514359765
222,0.12526978699255062,0.02001308388177821
354,0.12544364521781218,0.012683525845360586
426,0.12547443656204887,0.010933278925088838
618,0.12548166411265646,0.008218612898464133
910,0.12537299960788234,0.005624127488268119
1288,0.12538482302457316,0.003797127195811347
1430,0.12539049563648536,0.003336079583929909
1534,0.12539134653902204,0.003196725181132287
1962,0.1253943547146337,0.0024797182310986967
}
  \pgfplotstabletypeset[
    create on use/rate/.style={
      create col/gradient loglog={N}{eta},
    },
    clear infinite,
    col sep=comma,
    sci zerofill={true},
    precision=2,
    columns={N,norm,eta},
    columns/N/.style={
     column name={$N$},
     fixed,
    },
    columns/norm/.style={
     column name={$\|\U_h\|_1$},
     precision=6,
    },
    columns/eta/.style={
     column name={$\eta$},
    },
  ]{
N,norm,eta
2210,0.12538269510578096,0.0022436255610020967
2718,0.1253837647362229,0.0018152779301406317
2946,0.1253841176230748,0.0016180690164700228
3354,0.1253846887692514,0.0013590821327304791
3766,0.12538501613378153,0.0012251285300787706
4270,0.12538514067796114,0.0010852922449460485
5184,0.12538557792859234,0.0009096733948862308
5800,0.12538568587708857,0.000810006286386786
6376,0.12538575903681326,0.0007419458823618061
6640,0.12538578498315298,0.0007087119288495327
7946,0.1253856502670358,0.00060272364045692846
}
\end{table}

\begin{figure}[h!]
    \centering
    \includegraphics[width=0.6\textwidth]{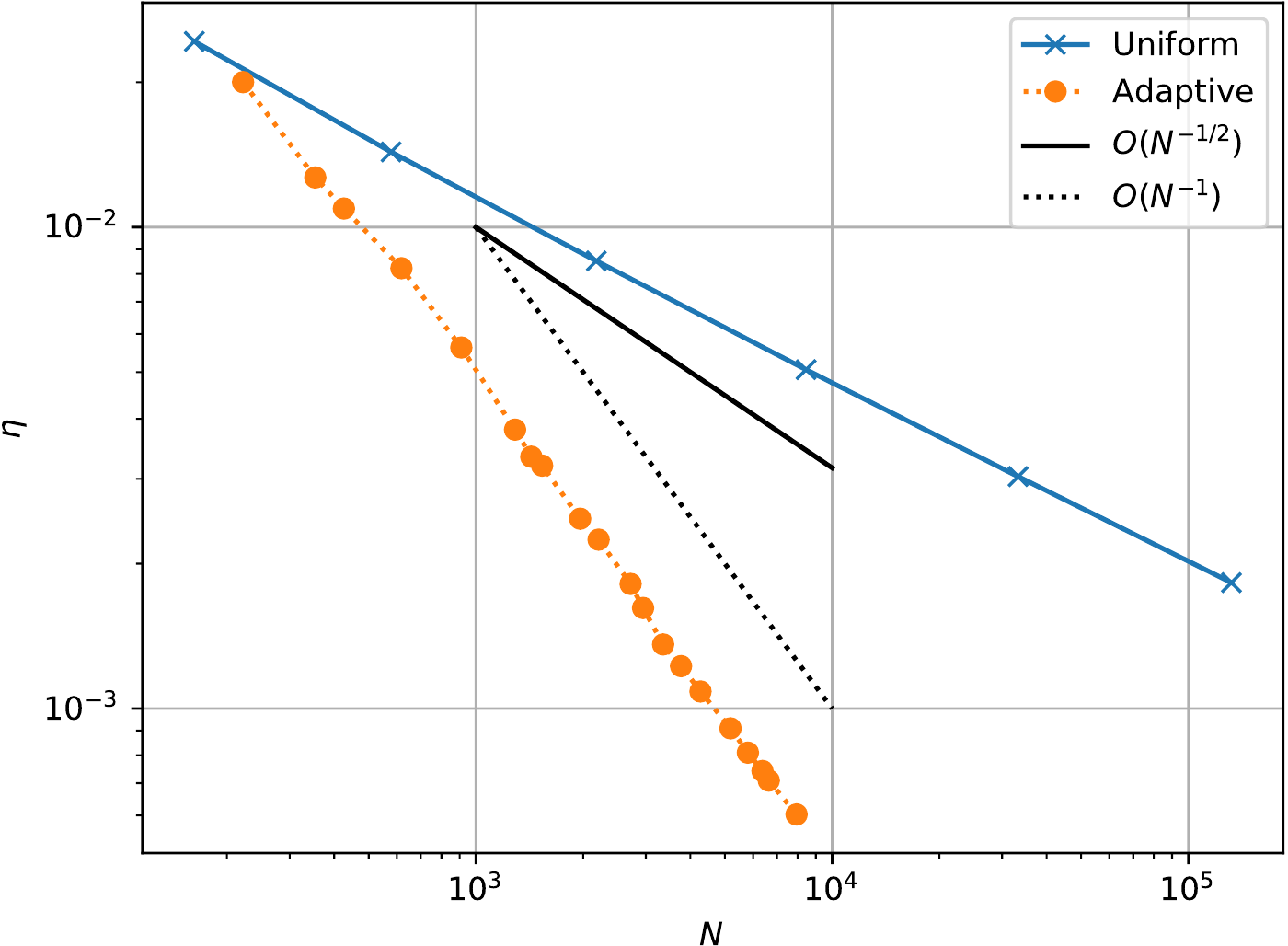}
    \caption{The total error indicator $\eta$ as a function of the number of degrees-of-freedom $N$.}
    \label{fig:ex1conv}
\end{figure}

\section*{Acknowledgements}

This work was supported by the Academy of Finland (Decisions 324611 and 338341) and by the Portuguese government through FCT (Funda\c{c}\~ao para a Ci\^encia e a Tecnologia), I.P., under the projects PTDC/MAT-PUR/28686/2017 and UIDB/04459/2020.

\bibliography{tresca}
\bibliographystyle{elsarticle-num}

\end{document}